\documentclass[12pt]{amsart}

\usepackage[utf8]{inputenc}
\usepackage[a4paper,nomarginpar]{geometry}
\usepackage[english]{babel}
\usepackage{enumitem}
\usepackage{amsmath,amsfonts,amssymb,amsthm,amscd}
\usepackage{tikz}
\usepackage{xcolor}

\allowdisplaybreaks

\theoremstyle{theorem}
\newtheorem{theorem}{Theorem}
\newtheorem{lemma}[theorem]{Lemma}
\newtheorem{proposition}[theorem]{Proposition}
\newtheorem{corollary}[theorem]{Corollary}

\theoremstyle{definition}

\newtheorem{example}[theorem]{Example}
\newtheorem{remark}[theorem]{Remark}

\theoremstyle{remark} \theoremstyle{question} \theoremstyle{example}

\newcommand{\N}{\mathbb{N}}
\newcommand{\Z}{\mathbb{Z}}

\newcommand{\R}{\mathbb{R}}
\newcommand{\C}{\mathbb{C}}

\newcommand{\K}{\mathbb{K}}

\newcommand{\cP}{\mathcal{P}}

\newcommand{\fM}{\mathfrak{M}}

\newcommand{\spa}{\operatorname{span}}

\newcommand{\supp}{\operatorname{supp}}

\newcommand{\NS}{\operatorname{NS}}

\newcommand{\eps}{\varepsilon}
\newcommand{\ov} {\overline}

\DeclareRobustCommand{\rchi}{{\mathpalette\irchi\relax}}
\newcommand{\irchi}[2]{\raisebox{\depth}{$#1\chi$}}

\begin{document}


\title[Li-Yorke Chaotic Weighted Composition Operators]{Li-Yorke Chaotic Weighted Composition Operators}

\date{}
\dedicatory{}
\maketitle

\begin{center}
{\sc Nilson C. Bernardes Jr.}

\medskip
Institut Universitari de Matem\`atica Pura i Aplicada\\
Universitat Polit\`ecnica de Val\`encia\\
Cam\'i de Vera S/N, Edifici 8E, Acces F, 4a Planta, Val\`encia, 46022, Spain\\
and\\
Departamento de Matem\'atica Aplicada, Instituto de Matem\'atica\\
Universidade Federal do Rio de Janeiro\\
Caixa Postal 68530, Rio de Janeiro, 21941-909, Brazil

\smallskip
{\it e-mail}: ncbernardesjr@gmail.com

\bigskip
{\sc Fernanda M. Vasconcellos}

\medskip
Instituto do Noroeste Fluminense de Educa\c c\~ao Superior\\
Universidade Federal Fluminense, Av.\ Jo\~ao Jasbick, s/n$^\text{o}$\\
Santo Ant\^onio de P\'adua, RJ 28470-000, Brazil

\smallskip
{\it e-mail}: fernandamv@id.uff.br
\end{center}

\begin{abstract}
We establish complete characterizations of the notion of Li-Yorke chaos for weighted composition operators 
on $C_0(X)$ spaces and on $L^p(\mu)$ spaces. 
As a consequence, we obtain simple characterizations of the Li-Yorke chaotic weighted shifts 
on $c_0$ and on $\ell^p$ ($1 \leq p < \infty$) that complement previously known results. 
We also investigate the notion of Li-Yorke chaos for weighted shifts on Fr\'echet sequence spaces.
As applications, we obtain characterizations of the Li-Yorke chaotic weighted shifts on K\"othe sequence spaces 
depending only on the entries of the K\"othe matrix and the weights of the shift.
\end{abstract}


\section{Introduction}

Chaos theory is a very popular branch of the area of dynamical systems. 
Among the many notions of chaos studied so far, Li-Yorke chaos has a prominent place, not only because it was the first notion of chaos to appear explicitly in the mathematical literature, but also because it motivated a whole class of interesting notions of chaos based on the behaviors of pairs of points (dense chaos, generic chaos, distributional chaos of types 1, 2 and 3, and so on). 
Moreover, it has interesting relationships with other fundamental notions in dynamics, including the notion of topological entropy.

\smallskip
Given a metric space $M$ with metric $d$, recall that a continuous map $f : M \to M$ is said to be {\em Li-Yorke chaotic} if there is an uncountable set $S \subset M$ (a {\em scrambled set} for $f$) such that each pair $(x,y)$ of distinct points in $S$ is a {\em Li-Yorke pair} for $f$, in the sense that
\[
\liminf_{n \to \infty} d(f^n(x),f^n(y)) = 0 \ \ \text{ and } \ \ \limsup_{n \to \infty} d(f^n(x),f^n(y)) > 0.
\]

On the other hand, the area of linear dynamics has witnessed a great development during the last few decades and many interesting connections to other branches of dyna\-mics, such as topological dynamics, ergodic theory, differentiable dynamics and symbolic dynamics, have been established. Moreover, especially through the study of the dynamics of semigroups of operators, many applications to other areas of science have been found, including applications to biology, engineering and physics. For an overview of the area of linear dynamics, we refer the reader to the books \cite{FBayEMat09,KGroAPer11} and to the more recent papers \cite{FBayEMat16,BerBonPer20,NBerAMes20,NBerAMes21,NBerAPer24,BesMenPerPui19,EDAnMMai23,GriMatMen21}, 
where many additional references can be found.

\smallskip
In the setting of linear dynamics, an extensive study of Li-Yorke chaos was developed in \cite{BerBonMarPer11,BerBonMulPer15}.
In particular, the following useful characterizations were obtained: 
{\it For any continuous linear operator $T$ on any Fr\'echet space $X$, the following assertions are equivalent: 
\begin{itemize}
\item [\rm (i)] $T$ is Li-Yorke chaotic;
\item [\rm (ii)] $T$ admits a {\em semi-irregular vector}, that is, a vector $x \in X$ such that the sequence $(T^n x)_{n \in \N}$
  does not converge to zero but has a subsequence converging to zero.
\item [\rm (iii)] $T$ admits an {\em irregular vector}, that is, a vector $x \in X$ such that the sequence $(T^n x)_{n \in \N}$
  is unbounded but has a subsequence converging to zero.
\end{itemize}}

Composition operators appear naturally in many different contexts (analytic, measure-theoretic, topological, etc) 
and so they actually constitute a family of classes of operators. 
They play a very important role in operator theory and its applications. 
The dynamics of such operators has been investigated by several authors in different settings 
(see \cite{BayDarPir18,LBerAMon95,BerBonMulPer13,BonKalPer21,BonDAnDarPia22,PBouJSha97,DAnDarMai21,
DAnDarMai22,UDarBPir21,EGalAMon04,KGro00,TKal07,TKal19,JSha93}, for instance). 
In particular, a detailed study of Li-Yorke chaos for composition operators on $L^p(\mu)$ spaces was developed in \cite{BerDarPir20}. 

\smallskip
In the present note we will investigate the notion of Li-Yorke chaos for weighted composition operators
\[
C_{w,f}(\varphi)\!:= (\varphi \circ f) \cdot w
\]
on $C_0(X)$ spaces (Section~\ref{Section2}) and on $L^p(\mu)$ spaces (Section~\ref{Section3}). 
Our goal is to obtain complete characterizations of Li-Yorke chaos for these operators (Theorems~\ref{MainLYCX} and~\ref{MainLY}). 
As applications, we will derive simple characterizations of the Li-Yorke chaotic weighted shifts (unilateral and bilateral) on the
classical Banach sequence spaces $c_0$ and $\ell^p$ (Corollaries~\ref{CorLYCX1}, \ref{CorLYCX2}, \ref{CorLY2} and~\ref{CorLY3}) 
that complement previously known results. 
In the case of $L^p(\mu)$ spaces, our results generalize previous results from \cite{BerDarPir20} in the unweighted case 
and we were strongly motivated by this previous work. 

\smallskip
In the final Section~\ref{Section4} we will study the notion of Li-Yorke chaos for weighted shifts on Fr\'echet sequence spaces.
As applications, we will obtain characterizations of the Li-Yorke chaotic weighted shifts on K\"othe sequence spaces 
that are of a computational character, that is, depend only on the entries of the K\"othe matrix and the weights of the shift
(Theorems~\ref{Kothe-LY-T1} and~\ref{Kothe-LY-T2}).

\smallskip
Throughout $\K$ will denote either the field $\R$ of real numbers or the field $\C$ of complex numbers.
Moreover, $\N$ will denote the set of positive integers $\{1,2,3,\ldots\}$.


\section{Weighted composition operators on $C_0(X)$ spaces}\label{Section2}

In this section we fix a locally compact Hausdorff space $X$ and let $C_0(X)$ denote the Banach space over $\K$ of all continuous maps 
$\varphi: X \to \K$ that vanish at infinity endowed with the supremum norm 
\[
\|\varphi\|\!:= \sup_{x \in X} |\varphi(x)|.
\]
More generally, for any subset $B$ of $X$, we define
\[
\|\varphi\|_B\!:= \sup_{x \in B} |\varphi(x)|,
\]
where we consider this supremum to be $0$ whenever $B = \varnothing$.
Given a map $\varphi : X \to \K$, we denote by $\supp \varphi$ its {\em support}, that is,
\[
\supp \varphi\!:= \ov{\{x \in X : \varphi(x) \neq 0\}}.
\]
We also fix a {\em weight function} $w : X \to \K$, that is, a continuous map such that
\begin{equation}\label{wf1}
\varphi \cdot w \in C_0(X) \ \text{ for all } \varphi \in C_0(X).
\end{equation}
It is not difficult to see that (\ref{wf1}) holds if and only if $w$ is bounded on $X$.
Given a continuous map $f : X \to X$, it is also not difficult to show that the weighted composition operator
\[
C_{w,f}(\varphi)\!:= (\varphi \circ f) \cdot w
\]
is a well-defined bounded linear operator on $C_0(X)$ if and only if
\begin{equation}\label{Condition1}
f^{-1}(K) \cap X_\eps \text{ is compact for every } \eps > 0 \text{ and } K \subset X \text{ compact}, 
\end{equation}
where $X_\eps\!:= \{x \in X : |w(x)| \geq \eps\}$. In particular, the (unweighted) composition operator
\[
C_f(\varphi)\!:= \varphi \circ f
\]
is a well-defined bounded linear operator on $C_0(X)$ if and only if 
\[
f^{-1}(K) \text{ is compact for every } K \subset X \text{ compact},
\]
that is, $f$ is a proper mapping.
Moreover, if $w \in C_0(X)$ then $C_{w,f}$ is a well-defined bounded linear operator on $C_0(X)$ for every continuous map $f : X \to X$.

\smallskip
For the remaining of this section, we fix a continuous map $f : X \to X$ satisfying condition (\ref{Condition1}). 
Hence, the weighted composition operator
\[
C_{w,f} : \varphi \in C_0(X) \mapsto (\varphi \circ f) \cdot w \in C_0(X)
\]
is a well-defined bounded linear operator. We associate to $w$ and $f$ the following sequence of continuous maps from $X$ into $\K$:
\[
w^{(1)}\!:= w \ \ \ \text{ and } \ \ \ w^{(n)}\!:= (w \circ f^{n-1}) \cdot\ldots\cdot (w \circ f) \cdot w \ \text{ for } n \geq 2.
\]
If $w$ is the constant function $1$, then so is $w^{(n)}$ for all $n \geq 1$.

\smallskip
We now establish the main result of this section, a characterization of the weighted composition operators on $C_0(X)$ that are Li-Yorke chaotic. 

\begin{theorem}\label{MainLYCX}
The weighted composition operator $C_{w,f}$ on $C_0(X)$ is Li-Yorke chaotic if and only if there is a sequence $(B_i)_{i \in \N}$ of relatively compact open subsets of $X$ such that the following conditions hold:
\begin{itemize}
\item [\rm (A)] There is an increasing sequence $(n_j)_{j \in \N}$ of positive integers such that
\[
\lim_{j \to \infty} \|w^{(n_j)}\|_{f^{-n_j}(B_i)} = 0 \ \ \text{ for all } i \in \N.
\]
\item [\rm (B)] $\displaystyle \sup\big\{\|w^{(n)}\|_{f^{-n}(B_i)} : i,n \in \N\big\} = \infty$.
\end{itemize}
\end{theorem}

\begin{proof}
Suppose that there is a sequence $(B_i)_{i \in \N}$ with all the properties described in the statement of the theorem.
For each $i \in \N$, let $Y_i\!:= \{\varphi \in C_0(X) : \supp \varphi \subset B_i\}$, which is a vector subspace of $C_0(X)$. 
Let $Y$ be the closed linear span of $\bigcup_{i=1}^\infty Y_i$ in $C_0(X)$. 
By condition~(A), for each $i \in \N$ and each $\varphi \in Y_i$,
\[
\|(C_{w,f})^{n_j}(\varphi)\| = \|(\varphi \circ f^{n_j}) \cdot w^{(n_j)}\|_{f^{-n_j}(B_i)} \leq \|\varphi\| \|w^{(n_j)}\|_{f^{-n_j}(B_i)}
 \to 0 \text{ as } j \to \infty.
\]
Hence, the set of all $\varphi \in Y$ whose orbit under $C_{w,f}$ has a subsequence converging to $0$ is dense in $Y$, 
and so it is residual in $Y$ \cite[Proposition~3]{BerBonMulPer15}. 
On the other hand, given $i,n \in \N$ with $f^{-n}(B_i) \neq \emptyset$, choose $a_{i,n} \in f^{-n}(B_i)$ such that
\[
|w^{(n)}(a_{i,n})| \geq \|w^{(n)}\|_{f^{-n}(B_i)} - 1.
\]
By Urysohn's lemma, there is a continuous map $\phi_{i,n} : X \to [0,1]$ such that
\[
\supp \phi_{i,n} \subset B_i \ \ \ \text{ and } \ \ \ \phi_{i,n}(f^n(a_{i,n})) = 1.
\]
Hence, $\phi_{i,n} \in Y_i \subset Y$, $\|\phi_{i,n}\| = 1$ and
\[
\|(C_{w,f})^n(\phi_{i,n})\| \geq |\phi_{i,n}(f^n(a_{i,n})) w^{(n)}(a_{i,n})| \geq \|w^{(n)}\|_{f^{-n}(B_i)} - 1.
\]
Thus, condition~(B) implies that the sequence $(\|(C_{w,f})^n|_Y\|)_{n \in \N}$ is unbounded, and so the Banach-Steinhaus theorem \cite[Theorem~2.5]{WRud91} guarantees that the set of all $\varphi \in Y$ whose orbit under $C_{w,f}$ is unbounded is also residual in $Y$. 
This shows the existence of an irregular vector for $C_{w,f}$, proving that $C_{w,f}$ is Li-Yorke chaotic.

\smallskip
Conversely, suppose that the operator $C_{w,f}$ is Li-Yorke chaotic. Then it admits an irregular vector $\psi \in C_0(X)$. 
For each $i \in \N$, let
\[
B_i\!:= \{x \in X : |\psi(x)| > i^{-1}\},
\]
which is a relatively compact open subset of $X$.  
Since $\psi$ is an irregular vector for $C_{w,f}$, there is an increasing sequence $(n_j)_{j \in \N}$ of positive integers such that
$\|(C_{w,f})^{n_j}(\psi)\| \to 0$ as $j \to \infty$. Since, for each $i \in \N$,
\[
\|(C_{w,f})^{n_j}(\psi)\| \geq \|(\psi \circ f^{n_j}) \cdot w^{(n_j)}\|_{f^{-n_j}(B_i)} \geq i^{-1} \|w^{(n_j)}\|_{f^{-n_j}(B_i)},
\]
we obtain condition (A). Now, if 
\[
\beta\!:= \sup\big\{\|w^{(n)}\|_{f^{-n}(B_i)} : i,n \in \N\big\} < \infty,
\]
then
\[
\|(C_{w,f})^n(\psi)\| = \sup_{i \in \N} \|(\psi \circ f^n) \cdot w^{(n)}\|_{f^{-n}(B_i)} \leq \beta \|\psi\| \ \text{ for all } n \in \N,
\]
contradicting the fact that $\psi$ is an irregular vector for $C_{w,f}$. This proves condition (B).
\end{proof}

\begin{corollary}
A composition operator $C_f$ on $C_0(X)$ is never Li-Yorke chaotic.
\end{corollary}

\begin{remark}
Note that the sequence $(B_i)_{i \in \N}$ constructed in the proof of the necessity of the conditions in Theorem~\ref{MainLYCX} 
has the following additional property:
\begin{itemize}
\item [(C)] $\ov{B_i} \subset B_{i+1}$ for all $i \in \N$.
\end{itemize}
If we include this additional condition in the statement of Theorem~\ref{MainLYCX}, then we can replace condition (A) by the following weaker condition:
\begin{itemize}
\item [(A')] $\displaystyle \liminf_{n \to \infty} \|w^{(n)}\|_{f^{-n}(B_i)} = 0$ for all $i \in \N$.
\end{itemize}
In other words: {\it The weighted composition operator $C_{w,f}$ on $C_0(X)$ is Li-Yorke chaotic if and only if there is a sequence 
$(B_i)_{i \in \N}$ of relatively compact open subsets of $X$ satisfying conditions (A'), (B) and (C).}
Indeed, since the necessity is clear, let us prove the sufficiency. If
\[
\limsup_{n \to \infty} \|w^{(n)}\|_{f^{-n}(B_j)} > 0 \ \ \text{ for a certain } j \in \N,
\]
take a continuous map $\psi : X \to [0,1]$ with $\supp \psi \subset B_{j+1}$ and $\psi = 1 \text{ on } \ov{B_j}$. Since
\[
\|w^{(n)}\|_{f^{-n}(B_j)} \leq \|(C_{w,f})^n(\psi)\| \leq \|w^{(n)}\|_{f^{-n}(B_{j+1})} \ \ \text{ for all } n \in \N,
\]
we see that $\psi$ is a semi-irregular vector for $C_{w,f}$, and so $C_{w,f}$ is Li-Yorke chaotic. Assume
\[
\lim_{n \to \infty} \|w^{(n)}\|_{f^{-n}(B_i)} = 0 \ \ \text{ for all } i \in \N.
\]
Then condition (A) holds with $(n_j)_{j \in \N}\!:= (n)_{n \in \N}$. 
Since we are assuming that condition (B) holds, $C_{w,f}$ is Li-Yorke chaotic by Theorem~\ref{MainLYCX}.
\end{remark}

\begin{corollary}\label{CorLYCX1}
A unilateral weighted backward shift
\[
B_w : (x_1,x_2,x_3,\ldots) \in c_0(\N) \mapsto (w_1x_2,w_2x_3,\ldots) \in c_0(\N),
\]
where $w\!:= (w_n)_{n \in \N}$ is a bounded sequence of scalars, is Li-Yorke chaotic if and only if
\begin{equation}\label{EqLYCX1}
\sup\{|w_i \cdots w_j| : i,j \in \N, i \leq j\} = \infty.
\end{equation}
\end{corollary}

\begin{proof}
Consider $X\!:= \N$ endowed with the discrete topology and $f : n \in \N \mapsto n+1 \in \N$.
Then $C_0(X) = c_0(\N)$ and $C_{w,f} = B_w$. 
Since, for each $B \subset X$ finite, $f^{-n}(B) = \varnothing$ for all sufficiently large $n$, we have that condition (A) is superfluous in the present case. If (\ref{EqLYCX1}) fails, that is, $\beta\!:= \sup\{|w_i \cdots w_j| : i,j \in \N, i \leq j\} < \infty$, then
\[
\|w^{(n)}\|_{f^{-n}(A)} = \sup_{i \in A, i > n} |w^{(n)}(i-n)| = \sup_{i \in A, i > n} |w_{i-n} \cdots w_{i-1}| \leq \beta
\]
for all $A \subset \N$, and so condition (B) fails no matter how we choose the sequence $(B_i)_{i \in \N}$.
On the other hand, if (\ref{EqLYCX1}) is true, then condition (B) holds with $B_i\!:= \{i\}$ for each $i \in \N$.
\end{proof}

The above corollary was established in \cite[Proposition~27]{BerBonMarPer11} in the case $w$ is a bounded sequence of nonzero scalars.

\begin{corollary}\label{CorLYCX2}
A bilateral weighted backward shift
\[
B_w : (x_n)_{n \in \Z} \in c_0(\Z) \mapsto (w_n x_{n+1})_{n \in \Z} \in c_0(\Z),
\]
where $w\!:= (w_n)_{n \in \Z}$ is a bounded sequence of nonzero scalars, is Li-Yorke chaotic if and only if the following conditions hold:
\begin{itemize}
\item [\rm (a)] $\displaystyle \liminf_{n \to \infty} |w_{-n} \cdots w_{-1}| = 0$.
\item [\rm (b)] $\sup\{|w_i \cdots w_j| : i,j \in \Z, i \leq j\} = \infty$.
\end{itemize}
\end{corollary}

\begin{proof}
Consider $X\!:= \Z$ endowed with the discrete topology and $f : n \in \Z \mapsto n+1 \in \Z$.
Then $C_0(X) = c_0(\Z)$ and $C_{w,f} = B_w$. Suppose that (a) and (b) hold. If
\[
\limsup_{n \to \infty} |w_{-n} \cdots w_{-1}| > 0,
\]
then $\rchi_{\{0\}}$ is a semi-irregular vector for $C_{w,f}$, and so $C_{w,f}$ is Li-Yorke chaotic. If
\[
\lim_{n \to \infty} |w_{-n} \cdots w_{-1}| = 0,
\]
then conditions (A) and (B) hold with $(n_j)_{j \in \N}\!:= (n)_{n \in \N}$ and $(B_i)_{i \in \N}$ an enumeration of the sets $\{k\}$ 
for $k \in \Z$, and so $C_{w,f}$ is Li-Yorke chaotic by Theorem~\ref{MainLYCX}.
Conversely, suppose that $C_{w,f}$ is Li-Yorke chaotic and let $(B_i)_{i \in \N}$ and $(n_j)_{j \in \N}$ be as in the statement of Theorem~\ref{MainLYCX}. By choosing $k \in B_i$ for some $i \in \N$, we have that
\[
|w_{-n_j + k} \cdots w_{-1 + k}| = \|w^{(n_j)}\|_{f^{-n_j}(\{k\})} \leq \|w^{(n_j)}\|_{f^{-n_j}(B_i)} \to 0 \text{ as } j \to \infty,
\]
which gives (a). If (b) is false, that is, $\beta\!:= \sup\{|w_i \cdots w_j| : i,j \in \Z, i \leq j\} < \infty$, then
\[
\|w^{(n)}\|_{f^{-n}(B_i)} = \sup_{k \in B_i} |w_{-n+k} \cdots w_{-1+k}| \leq \beta \ \ \text{ for all } i \in \N,
\]
contradicting condition (B).
\end{proof}

To the best of the authors knowledge, a characterization of the Li-Yorke chaotic bilateral weighted shifts on $c_0(\Z)$ has not been given before.

\begin{remark}\label{RemarkLYCX}
We can remove the hypothesis that the weights $w_n$ are {\em nonzero} in Corollary~\ref{CorLYCX2} provided we state the result as follows: 
{\it Consider a bilateral weighted backward shift
\[
B_w : (x_n)_{n \in \Z} \in c_0(\Z) \mapsto (w_n x_{n+1})_{n \in \Z} \in c_0(\Z),
\]
where $w\!:= (w_n)_{n \in \Z}$ is a bounded sequence of scalars. Let
\[
\beta^-\!:= \sup\{|w_i \cdots w_j| : i \leq j \leq -1\} \ \ \text{ and } \ \ \beta^+\!:= \sup\{|w_i \cdots w_j| : 0 \leq i \leq j\}.
\]
Then $B_w$ is Li-Yorke chaotic if and only if one of the following conditions hold:
\begin{itemize}
\item [\rm (I)]  $\beta^+ = \infty$ and $\,\displaystyle \liminf_{n \to \infty} |w_{-n} \cdots w_{-k}| = 0$ for some $k \in \Z$.
\item [\rm (II)] $\beta^+ < \infty$, $\beta^- = \infty$ and $\,\displaystyle \liminf_{n \to \infty} |w_{-n} \cdots w_{-k}| = 0$ for all $k \in \N$.
\end{itemize}}

Indeed, let $X$ and $f$ be as in the proof of Corollary~\ref{CorLYCX2}. Suppose that $C_{w,f}$ is Li-Yorke chaotic and let 
$(B_i)_{i \in \N}$ and $(n_j)_{j \in \N}$ be as in the statement of Theorem~\ref{MainLYCX}. Let 
\[
\beta\!:= \sup\{|w_i \cdots w_j| : i,j \in \Z, i \leq j\} \ \ \text{ and } \ \ B\!:= \bigcup_{i \in \N} B_i.
\]
We claim that
\begin{equation}\label{NotesLYCX1}
\liminf_{n \to \infty} |w_{-n} \cdots w_{-1 + \ell}| = 0 \ \ \text{ for all } \ell \in B.
\end{equation}
Indeed, take $\ell \in B_i$ for some $i \in \N$. Since
\[
|w_{-n_j + \ell} \cdots w_{-1 + \ell}| = \|w^{(n_j)}\|_{f^{-n_j}(\{\ell\})} \leq \|w^{(n_j)}\|_{f^{-n_j}(B_i)} \to 0 \ \text{ as } j \to \infty,
\]
we obtain (\ref{NotesLYCX1}). Moreover, we must have $\beta = \infty$. In fact, if $\beta < \infty$ then 
\[
\|w^{(n)}\|_{f^{-n}(B_i)} = \sup_{k \in B_i} |w_{-n+k} \cdots w_{-1+k}| \leq \beta \ \ \text{ for all } i \in \N,
\]
contradicting condition (B). If $\beta^+ = \infty$, then (\ref{NotesLYCX1}) implies that condition (I) holds.
Assume $\beta^+ < \infty$. Since $\beta \leq \max\{\beta^-, \beta^+, \beta^- \!\!\cdot\! \beta^+\}$, we must have $\beta^- = \infty$. 
If $B$ is unbounded below, then (\ref{NotesLYCX1}) implies that (II) holds. Assume $B$ bounded below and let $m\!:= \min B$. 
We claim that
\begin{equation}\label{NotesLYCX3}
w_n \neq 0 \ \ \text{ for all } n < m.
\end{equation}
Indeed, assume $w_t = 0$ for a certain $t < m$ and let $C\!:= \sup\{|w_i \cdots w_j| : t < i \leq j\}$.
Since $\beta^+ < \infty$, we have that $C < \infty$. Since $t < m = \min B$ and $w_t = 0$,
\[
\|w^{(n)}\|_{f^{-n}(B_i)} = \sup_{k \in B_i} |w_{-n+k} \cdots w_{-1+k}| \leq C \ \text{ for all } i \in \N.
\]
This contradiction proves our claim. By (\ref{NotesLYCX1}), $\liminf_{n \to \infty} |w_{-n} \cdots w_{-1 + m}| = 0$.
Hence, (\ref{NotesLYCX3}) implies that (II) holds.

\smallskip
We now prove the converse. We can replace the ``$\liminf$'' in (I) and (II) by ``$\lim$'', since
\[
\liminf_{n \to \infty} |w_{-n} \cdots w_{-k}| = 0 \ \ \text{ and } \ \ \limsup_{n \to \infty} |w_{-n} \cdots w_{-k}| > 0
\]
imply that $\rchi_{\{-k+1\}}$ is a semi-irregular vector for $C_{w,f}$, and so $C_{w,f}$ is Li-Yorke chaotic.

\smallskip
If (I) holds, let $B_i\!:= \{-k+i\}$ for $i \in \N$. Since $\lim_{n \to \infty} |w_{-n} \cdots w_{-k}| = 0$,
\[
\lim_{n \to \infty} \|w^{(n)}\|_{f^{-n}(B_i)} = \lim_{n \to \infty} |w_{-n-k+i} \cdots w_{-1-k+i}| = 0 \ \text{ for all } i \in \N.
\]
Moreover, condition (B) holds because $\beta^+ = \infty$. Thus, $C_{w,f}$ is Li-Yorke chaotic. 

\smallskip
If (II) holds, let $B_i\!:= \{-i\}$ for $i \in \N$. Since $\lim_{n \to \infty} |w_{-n} \cdots w_{-k}| = 0$ for all $k \in \N$, 
\[
\lim_{n \to \infty} \|w^{(n)}\|_{f^{-n}(B_i)} = \lim_{n \to \infty} |w_{-n-i} \cdots w_{-1-i}| = 0 \ \text{ for all } i \in \N.
\]
Since $\beta^- = \infty$, condition (B) holds. Thus, $C_{w,f}$ is Li-Yorke chaotic. 
\end{remark}


\section{Weighted composition operators on $L^p(\mu)$ spaces}\label{Section3}

In this section we fix $p \in [1,\infty)$ and an arbitrary positive measure space $(X,\fM,\mu)$.
As usual, $L^p(\mu)$ denotes the Banach space over $\K$ of (equivalence classes of) $\K$-valued $p$-integrable functions on $(X,\fM,\mu)$ endowed with the $p$-norm 
\[
\|\varphi\|_p\!:= \left(\int_X |\varphi|^p\, d\mu\right)^\frac{1}{p}.
\]
We also fix a {\em weight function} $w : X \to \K$, which now means a measurable map such that
\begin{equation}\label{wf2}
\varphi \cdot w \in L^p(\mu) \ \text{ for all } \varphi \in L^p(\mu).
\end{equation}
The next result characterizes the weight functions in the case $\mu$ is semifinite.
This fact is probably known, but we present it below for the sake of completeness.

\begin{proposition}
If the measure $\mu$ is semifinite, that is, every set with infinite measure contains a set with positive finite measure 
(in particular, if $\mu$ is $\sigma$-finite), then (\ref{wf2}) holds if and only if $w \in L^\infty(\mu)$.
\end{proposition}

\begin{proof}
Since the sufficiency of the condition is clear, let us prove its necessity. 
For this purpose, suppose that $w \not\in L^\infty(\mu)$. For each $n \in \N$, let
\[
A_n\!:= \{x \in X : n^2 \leq |w(x)| < (n+1)^2\} \in \fM.
\]
Let $I\!:= \{n \in \N : \mu(A_n) > 0\}$. Since $w \not\in L^\infty(\mu)$, $I$ is an infinite set.
For each $n \in I$, let $B_n \in \fM$ be such that $B_n \subset A_n$ and $0 < \mu(B_n) < \infty$. Define
\[
\varphi\!:= \sum_{n \in I} \Big(\frac{1}{n^{2p} \mu(B_n)}\Big)^\frac{1}{p} \rchi_{B_n}.
\]
Note that $\varphi \in L^p(\mu)$, since the above series is absolutely convergent:
\[
\sum_{n \in I} \Big\|\Big(\frac{1}{n^{2p} \mu(B_n)}\Big)^\frac{1}{p} \rchi_{B_n}\Big\|_p
= \sum_{n \in I} \Big(\int_X \frac{1}{n^{2p} \mu(B_n)} \rchi_{B_n} d\mu\Big)^\frac{1}{p}
= \sum_{n \in I} \frac{1}{n^2} < \infty.
\]
However,
\[
\int_X |\varphi \cdot w|^p d\mu = \sum_{n \in I} \frac{1}{n^{2p} \mu(B_n)} \int_{B_n} |w|^p d\mu
\geq \sum_{n \in I} \frac{1}{n^{2p} \mu(B_n)} \cdot n^{2p} \mu(B_n)
= \sum_{n \in I} 1 = \infty,
\]
proving that $\varphi \cdot w \not\in L^p(\mu)$. This shows that (\ref{wf2}) fails.
\end{proof}

\begin{remark}
The hypothesis that $\mu$ is semifinite cannot be omitted in the previous proposition. 
For instance, suppose that there is a sequence $(x_n)$ of pairwise distinct points in $X$ such that 
$\{x_n\} \in \fM$ and $\mu(\{x_n\}) = \infty$ for all $n \in \N$. 
Let $Y\!:= X \backslash \{x_1,x_2,\ldots\}$ and define $w : X \to \R$ by $w\!:= 1$ on $Y$ and $w(x_n)\!:= n$ for all $n \in \N$. 
Then $w \not\in L^\infty(\mu)$, but
\[
\varphi \in L^p(\mu) \ \ \Rightarrow \ \ \varphi(x_n) = 0 \ \forall n \in \N \ \ \Rightarrow \ \ \varphi \cdot w = \varphi \cdot \rchi_Y \in L^p(\mu).
\]
\end{remark}

Given a bimeasurable map $f : X \to X$ (i.e., $f(A) \in \fM$ and $f^{-1}(A) \in \fM$ for all $A \in \fM$), it is easy to show that the weighted composition operator
\[
C_{w,f}(\varphi)\!:= (\varphi \circ f) \cdot w
\]
is a well-defined bounded linear operator on $L^p(\mu)$ if and only if there exists a constant $c \in (0,\infty)$ such that
\begin{equation}\label{condition}
\int_A |w|^p d\mu \leq c\, \mu(f(A)) \ \textrm{ for every } A \in \fM.
\end{equation}
Moreover, in this case, $\|C_{w,f}\| \leq c^\frac{1}{p}$.

\smallskip
For the remaining of this section, we fix a bimeasurable map $f : X \to X$ satisfying condition (\ref{condition}).
Hence, the weighted composition operator
\[
C_{w,f} : \varphi \in L^p(\mu) \mapsto (\varphi \circ f) \cdot w \in L^p(\mu)
\]
is a well-defined bounded linear operator. If $w$ is the constant function $1$, then we obtain the (unweighted) composition operator
\[
C_f : \varphi \in L^p(\mu) \mapsto \varphi \circ f \in L^p(\mu).
\]
We associate to $\mu$, $w$ and $f$, the following sequence of positive measures on $(X,\fM)$:
\[
\mu_1(A)\!:= \int_A |w|^p d\mu, \ \ \mu_n(A)\!:= \int_A |w \circ f^{n-1}|^p \cdots |w \circ f|^p |w|^p d\mu \ \ (A \in \fM, n \geq 2).
\]
If $w$ is the constant function $1$, then $\mu_n = \mu$ for all $n \geq 1$.

\begin{lemma}\label{LemmaC}
$\mu_n(A) \leq c^n \mu(f^n(A))$ for all $A \in \fM$ and $n \in \N$.
\end{lemma}

\begin{proof}
The case $n = 1$ is obvious. If $n \geq 2$ and $A \in \fM$, then
\begin{align*}
\mu_n(A) &\leq \int_X |\rchi_{f^n(A)} \circ f^n|^p |w \circ f^{n-1}|^p \cdots |w \circ f|^p |w|^p d\mu\\
&= \|(C_{w,f})^n(\rchi_{f^n(A)})\|_p^p \leq c^n \mu(f^n(A)),
\end{align*}
as it was to be shown.
\end{proof}

We now establish the main result of this section, a characterization of the weighted composition operators on $L^p(\mu)$ that are Li-Yorke chaotic. 

\begin{theorem}\label{MainLY}
The weighted composition operator $C_{w,f}$ on $L^p(\mu)$ is Li-Yorke chaotic if and only if there is a nonempty countable family 
$(B_i)_{i \in I}$ of measurable sets of finite positive $\mu$-measure such that the following conditions hold:
\begin{itemize}
\item [\rm (A)] There is an increasing sequence $(n_j)_{j \in \N}$ of positive integers such that
\[
\lim_{j \to \infty} \mu_{n_j}(f^{-n_j}(B_i)) = 0 \ \ \text{ for all } i \in I.
\]
\item [\rm (B)] $\displaystyle \sup\Bigg\{\frac{\mu_n(f^{-n}(B_i))}{\mu(B_i)} : i \in I, n \in \N\Bigg\} = \infty$.
\end{itemize}
\end{theorem}

\begin{proof}
($\Leftarrow$): Let $Y$ be the closed linear span of $\{\rchi_{B_i} : i \in I\}$ in $L^p(\mu)$. By condition (A), for each $i \in I$,
\begin{align*}
\|(C_{w,f})^{n_j}(\rchi_{B_i})\|_p^p 
&= \int_X |\rchi_{B_i} \circ f^{n_j}|^p |w \circ f^{n_j-1}|^p \cdots |w \circ f|^p |w|^p d\mu\\
&= \int_{f^{-n_j}(B_i)} |w \circ f^{n_j-1}|^p \cdots |w \circ f|^p |w|^p d\mu\\
&= \mu_{n_j}\big(f^{-n_j}(B_i)\big) \to 0 \ \text{ as } j \to \infty.
\end{align*}
Hence, the set of all $\varphi \in Y$ whose orbit under $C_{w,f}$ has a subsequence converging to $0$ is dense in $Y$,
and so it is residual in $Y$ by \cite[Proposition~3]{BerBonMulPer15}. 
For each $i \in I$, let $\phi_i\!:= \mu(B_i)^{-\frac{1}{p}} \cdot \rchi_{B_i} \in Y$. Since
\[
\|\phi_i\|_p = 1 \ \ \text{ and } \ \ \|(C_{w,f})^n(\phi_i)\|_p^p = \frac{\mu_n\big(f^{-n}(B_i)\big)}{\mu(B_i)} \ \ \ \ (n \in \N),
\]
it follows from condition (B) that the sequence $(\|(C_{w,f})^n|_Y\|)_{n \in \N}$ is unbounded.
Hence, by the Banach-Steinhaus theorem \cite[Theorem~2.5]{WRud91}, 
the set of all $\varphi \in Y$ whose orbit under $C_{w,f}$ is unbounded is also residual in $Y$. 
Thus, $C_{w,f}$ admits an irregular vector, and so it is Li-Yorke chaotic.

\smallskip\noindent
($\Rightarrow$): Since the operator $C_{w,f}$ is Li-Yorke chaotic, it admits an irregular vector $\psi \in L^p(\mu)$. For each $i \in \Z$, let
\[
B_i\!:= \{x \in X : 2^{i-1} \leq |\psi(x)| < 2^i\} \in \fM.
\]
Consider the countable set $I\!:= \{i \in \Z : \mu(B_i) > 0\}$, which is nonempty because $\psi$ is not the zero vector in $L^p(\mu)$. 
Since 
\[
\sum_{i \in \Z} 2^{(i-1)p} \mu(B_i) \leq \int_X |\psi|^p d\mu < \infty,
\]
we have that $\mu(B_i) \in (0,\infty)$ for all  $i \in I$.
Since $\psi$ is an irregular vector for $C_{w,f}$, there is an increasing sequence $(n_j)_{j \in \N}$ of positive integers such that
$\|(C_{w,f})^{n_j}(\psi)\|_p \to 0$ as $j \to \infty$. Since, for each $i \in I$,
\begin{align*}
\|(C_{w,f})^{n_j}(\psi)\|_p^p &= \int_X |\psi \circ f^{n_j}|^p |w \circ f^{n_j-1}|^p \cdots |w \circ f|^p |w|^p d\mu\\
&\geq \int_{f^{-n_j}(B_i)} |\psi \circ f^{n_j}|^p |w \circ f^{n_j-1}|^p \cdots |w \circ f|^p |w|^p d\mu\\
&\geq 2^{(i-1)p} \mu_{n_j}\big(f^{-n_j}(B_i)\big),
\end{align*}
we obtain condition (A). If condition (B) is false, then there exists $\beta \in \R$ such that
\begin{equation}\label{EqLY1}
\mu_n\big(f^{-n}(B_i)\big) \leq \beta \mu(B_i) \ \text{ for all }  i \in I \text{ and } n \in \N.
\end{equation}
In view of Lemma~\ref{LemmaC}, we have that (\ref{EqLY1}) holds for all $i \in \Z$ and $n \in \N$. Therefore,
\begin{align*}
\|(C_{w,f})^n(\psi)\|_p^p 
&= \sum_{i \in \Z} \int_{f^{-n}(B_i)} |\psi \circ f^n|^p |w \circ f^{n-1}|^p \cdots |w \circ f|^p |w|^p d\mu\\
&\leq \sum_{i \in \Z} 2^{ip} \int_{f^{-n}(B_i)} |w \circ f^{n-1}|^p \cdots |w \circ f|^p |w|^p d\mu\\
&= \sum_{i \in \Z} 2^{ip} \mu_n\big(f^{-n}(B_i)\big)
\leq 2^p \beta \sum_{i \in \Z} 2^{(i-1)p} \mu(B_i)
\leq 2^p \beta \|\psi\|_p^p,
\end{align*}
which contradicts the fact that $\psi$ is an irregular vector for $C_{w,f}$.
\end{proof}

\begin{remark}\label{RemarkLY1}
Note that condition (A) in Theorem~\ref{MainLY} can be replaced by the following weaker condition:
\begin{itemize}
\item [\rm (A')] $\displaystyle \liminf_{n \to \infty} \mu_n(f^{-n}(B_i)) = 0$ for all $i \in I$.
\end{itemize}
In fact, suppose that there is a nonempty countable family $(B_i)_{i \in I}$ of measurable sets of finite positive $\mu$-measure satisfying conditions (A') and (B). We have to prove that $C_{w,f}$ is Li-Yorke chaotic. If
\[
\limsup_{n \to \infty} \mu_n(f^{-n}(B_k)) > 0 \ \ \text{ for some } k \in I,
\]
then $\rchi_{B_k}$ is a semi-irregular vector for $C_{w,f}$, and so $C_{w,f}$ is Li-Yorke chaotic. Assume
\[
\lim_{n \to \infty} \mu_n(f^{-n}(B_i)) = 0 \ \ \text{ for all } i \in I.
\]
This means that condition (A) holds with $(n_j)_{j \in \N}\!:= (n)_{n \in \N}$. 
Since we are assuming that condition (B) holds, $C_{w,f}$ is Li-Yorke chaotic by Theorem~\ref{MainLY}.
\end{remark}

As an immediate consequence of Theorem~\ref{MainLY}, we obtain the following result, which is exactly Theorem~1.1 in \cite{BerDarPir20}.

\begin{corollary}\label{CorLY1}
The composition operator $C_f$ on $L^p(\mu)$ is Li-Yorke chaotic if and only if there is a nonempty countable family $(B_i)_{i \in I}$ of measurable sets of finite positive $\mu$-measure such that the following conditions hold:
\begin{itemize}
\item There is an increasing sequence $(n_j)_{j \in \N}$ of positive integers such that
\[
\lim_{j \to \infty} \mu(f^{-n_j}(B_i)) = 0 \ \ \text{ for all } i \in I.
\]
\item $\displaystyle \sup\Bigg\{\frac{\mu(f^{-n}(B_i))}{\mu(B_i)} : i \in I, n \in \N\Bigg\} = \infty$.
\end{itemize}
\end{corollary}

\begin{remark}\label{RemarkLY2}
In view of Remark~\ref{RemarkLY1}, we have that the first condition in the above corollary can be replaced by the following weaker condition:
\begin{itemize}
\item $\displaystyle \liminf_{n \to \infty} \mu(f^{-n}(B_i)) = 0$ for all $i \in I$.
\end{itemize}
\end{remark}

\begin{corollary}\label{CorLY2}
A unilateral weighted backward shift
\[
B_w : (x_1,x_2,x_3,\ldots) \in \ell^p(\N) \mapsto (w_1x_2,w_2x_3,\ldots) \in \ell^p(\N),
\]
where $w\!:= (w_n)_{n \in \N}$ is a bounded sequence of scalars, is Li-Yorke chaotic if and only if
\begin{equation}\label{EqLY4}
\sup\{|w_i \cdots w_j| : i,j \in \N, i \leq j\} = \infty.
\end{equation}
\end{corollary}

\begin{proof}
Consider $X\!:= \N$, $\fM\!:= \cP(X)$ (the power set of $X$), $\mu$ the counting measure on $\fM$, and $f : n \in \N \mapsto n+1 \in \N$.
Then $L^p(\mu) = \ell^p(\N)$ and $C_{w,f} = B_w$. In the present case, condition (A) is superfluous.
If (\ref{EqLY4}) fails, that is, $\beta\!:= \sup\{|w_i \cdots w_j| : i,j \in \N, i \leq j\} < \infty$, then
\[
\mu_n(A) = \sum_{i \in A} |w_i \cdots w_{i+n-1}|^p \leq \beta^p \mu(A) \ \ \text{ for all } A \in \fM,
\]
and so condition (B) fails no matter how we choose the family $(B_i)_{i \in I}$.
On the other hand, if (\ref{EqLY4}) is true, then condition (B) holds with $B_i\!:= \{i\}$ for each $i \in \N$.
\end{proof}

The above corollary was established in \cite[Proposition~27]{BerBonMarPer11} in the case $w$ is a bounded sequence of nonzero scalars.

\begin{corollary}\label{CorLY3}
A bilateral weighted backward shift
\[
B_w : (x_n)_{n \in \Z} \in \ell^p(\Z) \mapsto (w_n x_{n+1})_{n \in \Z} \in \ell^p(\Z),
\]
where $w\!:= (w_n)_{n \in \Z}$ is a bounded sequence of nonzero scalars, is Li-Yorke chaotic if and only if the following conditions hold:
\begin{itemize}
\item [\rm (a)] $\displaystyle \liminf_{n \to \infty} |w_{-n} \cdots w_{-1}| = 0$.
\item [\rm (b)] $\sup\{|w_i \cdots w_j| : i,j \in \Z, i \leq j\} = \infty$.
\end{itemize}
\end{corollary}

\begin{proof}
Consider $X\!:= \Z$, $\fM\!:= \cP(X)$, $\mu$ the counting measure on $\fM$, and $f : n \in \Z \mapsto n+1 \in \Z$.
Then $L^p(\mu) = \ell^p(\Z)$ and $C_{w,f} = B_w$. Suppose that (a) and (b) hold. If
\[
\limsup_{n \to \infty} |w_{-n} \cdots w_{-1}| > 0,
\]
then $\rchi_{\{0\}}$ is a semi-irregular vector for $C_{w,f}$, and so $C_{w,f}$ is Li-Yorke chaotic. If
\[
\lim_{n \to \infty} |w_{-n} \cdots w_{-1}| = 0,
\]
then conditions (A) and (B) hold with $(n_j)_{j \in \N}\!:= (n)_{n \in \N}$ and $B_i\!:= \{i\}$ for all $i \in \Z$, and so $C_{w,f}$ is Li-Yorke chaotic by Theorem~\ref{MainLY}.
Conversely, suppose that $C_{w,f}$ is Li-Yorke chaotic and let $(B_i)_{i \in I}$ and $(n_j)_{j \in \N}$ be as in the statement of Theorem~\ref{MainLY}.
By choosing $i \in I$ and $k \in B_i$, we have that
\[
|w_{-n_j + k} \cdots w_{-1 + k}|^p = \mu_{n_j}(f^{-n_j}(\{k\})) \leq \mu_{n_j}(f^{-n_j}(B_i)) \to 0 \text{ as } j \to \infty,
\]
which gives (a). If (b) is false, that is, $\beta\!:= \sup\{|w_i \cdots w_j| : i,j \in \Z, i \leq j\} < \infty$, then
\[
\mu_n(f^{-n}(B_i)) = \sum_{k \in B_i} |w_{-n+k} \cdots w_{-1+k}|^p \leq \beta^p \mu(B_i) \ \ \text{ for all } i \in I,
\]
contradicting condition (B).
\end{proof}

The above corollary was established in \cite[Corollary~1.6]{BerDarPir20} in the case $w$ is a bounded sequence of positive scalars.

\begin{remark}\label{RemarkLY3}
We can remove the hypothesis that the weights $w_n$ are {\em nonzero} in Corollary~\ref{CorLY3} provided we state the result as follows: 
{\it Consider a bilateral weighted backward shift
\[
B_w : (x_n)_{n \in \Z} \in \ell^p(\Z) \mapsto (w_n x_{n+1})_{n \in \Z} \in \ell^p(\Z),
\]
where $w\!:= (w_n)_{n \in \Z}$ is a bounded sequence of scalars. Let
\[
\beta^-\!:= \sup\{|w_i \cdots w_j| : i \leq j \leq -1\} \ \ \text{ and } \ \ \beta^+\!:= \sup\{|w_i \cdots w_j| : 0 \leq i \leq j\}.
\]
Then $B_w$ is Li-Yorke chaotic if and only if one of the following conditions hold:
\begin{itemize}
\item [\rm (I)]  $\beta^+ = \infty$ and $\,\displaystyle \liminf_{n \to \infty} |w_{-n} \cdots w_{-k}| = 0$ for some $k \in \Z$.
\item [\rm (II)] $\beta^+ < \infty$, $\beta^- = \infty$ and $\,\displaystyle \liminf_{n \to \infty} |w_{-n} \cdots w_{-k}| = 0$ for all $k \in \N$.
\end{itemize}}

Indeed, let $X$, $\fM$, $\mu$ and $f$ be as in the proof of Corollary~\ref{CorLY3}.  
Assume $C_{w,f}$ Li-Yorke chaotic and let $(B_i)_{i \in I}$ and $(n_j)_{j \in \N}$ be as in Theorem~\ref{MainLY}.
Define
\[
\beta\!:= \sup\{|w_i \cdots w_j| : i,j \in \Z, i \leq j\} \ \ \text{ and } \ \ B\!:= \bigcup_{i \in I} B_i.
\]
If $\ell \in B_i$ for some $i \in I$, then
\[
|w_{-n_j + \ell} \cdots w_{-1 + \ell}|^p = \mu_{n_j}(f^{-n_j}(\{\ell\})) \leq \mu_{n_j}(f^{-n_j}(B_i)) \to 0 \ \text{ as } j \to \infty.
\]
Therefore,
\begin{equation}\label{NotesLY1}
\liminf_{n \to \infty} |w_{-n} \cdots w_{-1 + \ell}| = 0 \ \ \text{ for all } \ell \in B.
\end{equation}
Since 
\[
\mu_n(f^{-n}(B_i)) = \sum_{k \in B_i} |w_{-n+k} \cdots w_{-1+k}|^p \leq \beta^p \mu(B_i) \ \ \ \ (i \in I, n \in \N),
\]
condition (B) gives $\beta = \infty$. 
If $\beta^+ = \infty$, then (\ref{NotesLY1}) implies that (I) holds. Assume $\beta^+ < \infty$. 
Since $\beta \leq \max\{\beta^-, \beta^+, \beta^- \!\!\cdot\!\beta^+\}$, we must have $\beta^- = \infty$. 
If $B$ is unbounded below, (\ref{NotesLY1}) implies that (II) holds.
Assume $B$ bounded below and let $m\!:= \min B$. We claim that
\begin{equation}\label{NotesLY3}
w_n \neq 0 \ \ \text{ for all } n < m.
\end{equation}
Indeed, suppose that $w_t = 0$ for a certain $t < m$. Let
\[
C\!:= \sup\{|w_i \cdots w_j| : t < i \leq j\}.
\]
Since $\beta^+ < \infty$, we have that $C < \infty$. Since $t < m = \min B$ and $w_t = 0$,
\[
\mu_n(f^{-n}(B_i)) = \sum_{k \in B_i} |w_{-n+k} \cdots w_{-1+k}|^p \leq C^p \mu(B_i) \ \ \ \ (i \in I, n \in \N).
\]
This contradiction proves our claim. By (\ref{NotesLY1}),
\[
\liminf_{n \to \infty} |w_{-n} \cdots w_{-1 + m}| = 0.
\]
Hence, (\ref{NotesLY3}) implies that (II) holds.

\smallskip
Let us now prove the converse. If (I) holds, let $I\!:= \{i \in \Z : i > -k\}$ and $B_i\!:= \{i\}$ for each $i \in I$.
Since $\liminf_{n \to \infty} |w_{-n} \cdots w_{-k}| = 0$, we have that
\[
\liminf_{n \to \infty} \mu_n(f^{-n}(B_i)) = \liminf_{n \to \infty} |w_{-n+i} \cdots w_{-1+i}|^p = 0 \ \text{ for all } i \in I.
\]
Moreover, condition (B) holds because $\beta^+ = \infty$. 
Thus, in view of Remark~\ref{RemarkLY1}, we have that $C_{w,f}$ is Li-Yorke chaotic. 

Finally, if (II) holds, consider $I\!:= \{i \in \Z : i \leq -1\}$ and $B_i\!:= \{i\}$ for each $i \in I$. 
Since $\liminf_{n \to \infty} |w_{-n} \cdots w_{-k}| = 0$ for all $k \in \N$, we have that
\[
\liminf_{n \to \infty} \mu_n(f^{-n}(B_i)) = \liminf_{n \to \infty} |w_{-n+i} \cdots w_{-1+i}|^p = 0 \ \text{ for all } i \in I.
\]
Since $\beta^- = \infty$, condition (B) holds. Thus, $C_{w,f}$ is Li-Yorke chaotic by Remark~\ref{RemarkLY1}. 
\end{remark}


\section{Weighted shifts on Fr\'echet sequence spaces}\label{Section4}

In this section we complement the previous results on Li-Yorke chaos for weighted shifts on the classical Banach sequence spaces 
$c_0$ and $\ell^p$ by considering weighted shifts on the more general setting of abstract Fr\'echet sequence spaces.
Recall that a {\em Fr\'echet sequence space} is a Fr\'echet space $X$ which is a vector subspace of the product space $\K^\N$
such that the inclusion map $X \to \K^\N$ is continuous, i.e., convergence in $X$ implies coordinatewise convergence.
If $w\!:= (w_n)_{n \in \N}$ is a sequence of nonzero scalars, the closed graph theorem implies that the 
{\em unilateral weighted backward shift}
\[
B_w(x_1,x_2,x_3,\ldots)\!:= (w_1x_2,w_2x_3,w_3x_4,\ldots)
\]
is a continuous linear operator on $X$ as soon as it maps $X$ into itself.
The canonical vectors $e_n\!:= (\delta_{n,j})_{j \in \N} \in \K^\N$ ($n \in \N$) form a {\em basis} of $X$ if they belong to $X$ and
\[
x = \sum_{n=1}^\infty x_n e_n \ \ \text{ for all } x\!:= (x_n)_{n \in \N} \in X.
\]
Recall also that a {\em Fr\'echet sequence space over $\Z$} is a Fr\'echet space $X$ which is a vector subspace of the product space 
$\K^\Z$ such that the inclusion map $X \to \K^\Z$ is continuous. As before, if $w\!:= (w_n)_{n \in \Z}$ is a sequence of nonzero scalars,
then the {\em bilateral weighted backward shift}
\[
B_w((x_n)_{n \in \Z})\!:= (w_n x_{n+1})_{n \in \Z}
\]
is a continuous linear operator on $X$ as soon as it maps $X$ into itself.
The canonical vectors $e_n\!:= (\delta_{n,j})_{j \in \Z} \in \K^\Z$ ($n \in \Z$) form a {\em basis} of $X$ if they belong to $X$ and
\[
x = \sum_{n=-\infty}^\infty x_n e_n \ \ \text{ for all } x\!:= (x_n)_{n \in \Z} \in X.
\]

Finally, recall that a continuous map $f : M \to M$, where $M$ is a metric space, is said to be {\em densely} (resp.\ {\em generically}) 
{\em Li-Yorke chaotic} if it admits a dense (resp.\ residual) scrambled set.

\smallskip
The following trichotomy complements the study of Li-Yorke chaos for linear operators developed in \cite{BerBonMulPer15}.

\begin{theorem}\label{Trichotomy}
Let $T$ be a continuous linear operator on a Fr\'echet space $X$. If
\[
\NS(T)\!:= \{x \in X : (T^n x)_{n \in \N} \text{ has a subsequence converging to zero}\},
\]
then either
\begin{itemize}
\item [\rm (a)] $T$ has a residual set of irregular vectors, or
\item [\rm (b)] $\displaystyle \lim_{n \to \infty} T^n x = 0$ for every $x \in X$, or
\item [\rm (c)] $\NS(T)$ is not dense in $X$.
\end{itemize}
Moreover:
\begin{itemize}
\item [\rm (A)] If $\NS(T)$ is not closed in $X$, then $T$ is Li-Yorke chaotic.
\item [\rm (B)] If $X$ is separable and $\NS(T)$ is dense and not closed in $X$, then $T$ is densely Li-Yorke chaotic.
\end{itemize}
\end{theorem}

\begin{proof}
It is clear that the three properties (a), (b) and (c) are mutually exclusive. 

We assume that $\NS(T)$ is dense in $X$ and prove that either (a) or (b) must be true.
If $T$ has an irregular vector, then the Dense Li-Yorke Chaos Criterion \cite[Definition~16]{BerBonMulPer15} holds.
Hence, by \cite[Theorem~17]{BerBonMulPer15} (see also the last paragraph of \cite[Section~6]{BerBonMulPer15}), (a) is true.
If $T$ does not have an irregular vector, then \cite[Theorem~8]{BerBonMulPer15} guarantees that $T$ does not have a semi-irregular 
vector either. Thus,
\begin{equation}\label{cz}
\NS(T) = \Big\{x \in X : \lim_{n \to \infty} T^n x = 0\Big\},
\end{equation}
which is a vector subspace of $X$. Since $\NS(T)$ is residual in $X$, we must have $\NS(T) = X$. 
Indeed, given $x \in X$, since $\NS(T)$ and $x + \NS(T)$ are both residual in $X$, the intersection $\NS(T) \cap (x + \NS(T))$
is nonempty, and so $x \in \NS(T) + \NS(T) = \NS(T)$. Hence, (b) is true.

In order to prove (A), we assume that $T$ is not Li-Yorke chaotic. 
Let $M$ be the set in the right-hand side of (\ref{cz}) and let $Y\!:= \ov{M}$.
Consider the operator $S : Y \to Y$ obtained by restricting $T$ to $Y$. 
Since $T$ is not Li-Yorke chaotic, $T$ has no semi-irregular vector, and so $\NS(T) = M$. 
Moreover, both (a) and (c) are false for the operator $S$. 
Thus, (b) must hold for the operator $S$, which means that $M = Y = \ov{M}$.

Finally, assume the hypotheses of (B). 
Since $\NS(T)$ is dense in $X$, (c) does not hold.
Since $\NS(T)$ is not closed in $X$, (A) ensures that $T$ is Li-Yorke chaotic, and so (b) is also false. 
By the above trichotomy, (a) must be true.
Since we are assuming that $X$ is separable, $T$ is densely Li-Yorke chaotic by \cite[Theorem~10]{BerBonMulPer15}.
\end{proof}

\begin{remark}
The converses of (A) and (B) are false in general. For instance, for every generically Li-Yorke chaotic operator $T$, 
we have that $\NS(T) = X$ by \cite[Theorem~34]{BerBonMulPer15}.
\end{remark}

\begin{corollary}\label{UWBS}
Suppose that $X$ is a Fr\'echet sequence space in which the sequence $(e_n)_{n \in \N}$ of canonical vectors is a basis,
$w\!:= (w_n)_{n \in \N}$ is a sequence of nonzero scalars, and the unilateral weighted backward shift $B_w$ is a well-defined 
operator on $X$. Then either
\begin{itemize}
\item [\rm (a)] $B_w$ is densely Li-Yorke chaotic, or
\item [\rm (b)] $\displaystyle \lim_{n \to \infty} (B_w)^n x = 0$ for every $x \in X$.
\end{itemize}
\end{corollary}

\begin{proof}
Since $\NS(B_w)$ is dense in $X$, the result follows from the above trichotomy.
\end{proof}

In other words, the unilateral weighted backward shift $B_w$ on $X$ is Li-Yorke chaotic
if and only if there is a vector $(x_n)_{n \in \N} \in X$ such that
\[
(w_1 \cdots w_n x_{n+1},w_2 \cdots w_{n+1} x_{n+2},w_3 \cdots w_{n+2} x_{n+3},\ldots) \not\to 0 \text{ in } X \text{ as } n \to \infty.
\]
In particular, the unilateral unweighted backward shift $B$ on $X$ is Li-Yorke chaotic (provided it is well-defined, of course)
if and only if there is a vector $(x_n)_{n \in \N} \in X$ such that
\[
(x_{n+1},x_{n+2},x_{n+3},\ldots) \not\to 0 \text{ in } X \text{ as } n \to \infty.
\]
For example, this is not the case if $X = c_0(\N)$ or $\ell^p(\N)$ for any $p \in [1,\infty)$, but it is true if $X = \K^\N$.

\begin{corollary}\label{BWBS}
Suppose that $X$ is a Fr\'echet sequence space over $\Z$ in which the sequence $(e_n)_{n \in \Z}$ of canonical vectors is a basis,
$w\!:= (w_n)_{n \in \Z}$ is a sequence of nonzero scalars, and the bilateral weighted backward shift $B_w$ is a well-defined 
operator on $X$. If the sequence $(w_{-n} \cdots w_{-1} e_{-n})_{n \in \N}$ has a subsequence converging to zero, then either
\begin{itemize}
\item [\rm (a)] $B_w$ is densely Li-Yorke chaotic, or
\item [\rm (b)] $\displaystyle \lim_{n \to \infty} (B_w)^n x = 0$ for every $x \in X$.
\end{itemize}
\end{corollary}

\begin{proof}
The additional hypothesis on the weights means that there is an increasing sequence $(n_j)_{j \in \N}$ of positive integers such that
\[
(B_w)^{n_j}e_0 \to 0 \ \text{ as } j \to \infty.
\]
For each $k \geq 1$, there is a constant $a_k$ such that
\[
(B_w)^k e_k = a_k e_0.
\]
For each $k \leq -1$, there is a constant $b_k$ such that
\[
(B_w)^{-k} e_0 = b_k e_k.
\]
Thus, for every $k \in \Z$,
\begin{equation}
(B_w)^{n_j + k} e_k \to 0 \ \text{ as } j \to \infty.\label{limit}
\end{equation}
Given $m \geq 1$ and $k \in \{-m,\ldots,m\}$, by applying $(B_w)^{m-k}$ in (\ref{limit}), we see that
\[
(B_w)^{n_j + m} e_k \to 0 \ \text{ as } j \to \infty.
\]
Thus, $(B_w)^{n_j + m} x \to 0$ as $j \to \infty$, for every $x \in \spa(\{e_{-m},\ldots,e_m\})$. 
This proves that the set $\NS(B_w)$ contains every bilateral sequence of finite support, and so it is dense in~$X$.
Thus, the result follows from the previous trichotomy.
\end{proof}

\begin{remark}
The additional hypothesis on the weights is essential for the validity of the above corollary.
Indeed, for the unweighted backward shift on $\ell^2(\Z)$, both (a) and (b) are false.
\end{remark}

\begin{theorem}\label{BWBSThm}
Suppose that $X$ is a Fr\'echet sequence space over $\Z$ in which the sequence $(e_n)_{n \in \Z}$ of canonical vectors is a basis,
$w\!:= (w_n)_{n \in \Z}$ is a sequence of nonzero scalars, and the bilateral weighted backward shift $B_w$ is a well-defined 
operator on $X$. Then $B_w$ is Li-Yorke chaotic if and only if the following conditions hold:
\begin{itemize}
\item [\rm (a)] The sequence $(w_{-n} \cdots w_{-1} e_{-n})_{n \in \N}$ has a subsequence converging to zero.
\item [\rm (b)] There is a vector $(x_n)_{n \in \Z} \in X$ such that
\[
(w_j \cdots w_{j+n-1} x_{j+n})_{j \in \Z} \not\to 0 \text{ in } X \text{ as } n \to \infty.
\]
\end{itemize}
\end{theorem}

\begin{proof}
Suppose that $B_w$ is Li-Yorke chaotic and let $x\!:= (x_n)_{n \in \Z} \in X$ be a semi-irregular vector for $B_w$.
Since $(B_w)^n(x) \not\to 0$ as $n \to \infty$, (b) holds.
Choose $k \in \Z$ such that $x_k \neq 0$ and let $(n_j)_{j \in \N}$ be an increasing sequence of positive integers such that
$(B_w)^{n_j}(x) \to 0$ as $j \to \infty$.
Given a neighborhood $V$ of $0$ in $X$, the equicontinuity of the family of operators $y \in X \mapsto y_n e_n \in X$, $n \in \Z$,
implies the existence of a neighborhood $U$ of $0$ in $X$ such that
\[
y \in U \ \Longrightarrow \ y_ne_n \in V \text{ for all } n \in \Z.
\]
There exists $j_0 \in \N$ such that
\[
j \geq j_0 \ \Longrightarrow \ (B_w)^{n_j}(x) \in U \ \Longrightarrow \ w_{k-n_j} \cdots w_{k-1} x_k e_{k-n_j} \in V.
\]
This proves that $w_{k-n_j} \cdots w_{k-1} x_k e_{k-n_j} \to 0$ as $j \to \infty$, which implies (a).

Conversely, if (a) and (b) hold, then $B_w$ is (densely) Li-Yorke chaotic by the previous corollary.
\end{proof}

In particular, the bilateral unweighted backward shift $B$ on $X$ is Li-Yorke chaotic (provided it is well-defined, of course)
if and only if the sequence $(e_{-n})_{n \in \N}$ has a subsequence converging to zero and there is a vector 
$(x_n)_{n \in \Z} \in X$ such that
\[
(x_{j+n})_{j \in \Z} \not\to 0 \text{ in } X \text{ as } n \to \infty.
\]
For example, this is not the case if $X = c_0(\Z)$ or $\ell^p(\Z)$ for any $p \in [1,\infty)$, but it is true if $X = \K^\Z$.

\smallskip
It is known that $B$ is {\em hypercyclic} (i.e., it has a dense orbit) if and only if 
there is an increasing sequence $(n_j)_{j \in \N}$ of positive integers such that, for any $k \in \Z$,
\[
e_{k-n_j} \to 0 \ \text{ and } \ e_{k+n_j} \to 0 \ \text{ in } X \text{ as } j \to \infty
\]
(see \cite[Theorem~4.12]{KGroAPer11}).
With these characterizations it is easy to exhibit such a shift $B$ which is Li-Yorke chaotic but not hypercyclic
(although stronger and more sophisticated examples can be found in \cite{MarOprPer13}).
As an illustration, consider the weight sequence $\nu\!:= (\nu_n)_{n \in \Z}$, where $\nu_{-n}\!:= 1/n$ for all $n \geq 1$, and 
\[
(\nu_n)_{n \geq 0}\!:= (1,2,1,2,2^2,2,1,2,2^2,2^3,2^2,2,1,\ldots,1,2,2^2,\ldots,2^n,\ldots,2^2,2,1,\ldots),
\]
a concatenation of blocks of the form $(1,2,2^2,\ldots,2^n,\ldots,2^2,2)$ for $n \geq 1$.
Fix $p \in [1,\infty)$ and consider the weighted $\ell^p$ space
\[
\ell^p(\Z,\nu)\!:= \Big\{x\!:= (x_n)_{n \in \Z} \in \K^\Z : \|x\|_p^p\!:= \sum_{n=-\infty}^\infty |x_n|^p \nu_n < \infty\Big\}.
\]
Clearly the bilateral backward shift $B$ is a well-defined operator on $\ell^p(\Z,\nu)$.
It is not hypercyclic because $\|e_n\|_p \geq 1$ for all $n \geq 0$.
In order to see that it is Li-Yorke chaotic, let $j_n\!:= n(n+1)$ for each $n \in \N$, and consider the vector
\[
x\!:= \sum_{n=1}^\infty 2^{-\frac{n}{p}} e_{j_n} \in \ell^p(\Z,\nu).
\]
Since $e_{-n} \to 0$ as $n \to \infty$ and $\|B^n x\|_p \geq \|B^n(2^{-\frac{n}{p}} e_{j_n})\|_p = 1$ for all $n \in \N$,
we have that $B$ is Li-Yorke chaotic.

\smallskip
We shall now focus on a more concrete class of Fr\'echet sequence spaces known as K\"othe sequence spaces.
Let $J\!:= \N$ or $\Z$. Recall that a {\em K\"othe matrix} ({\em on $J$}) is a family of scalars of the form 
$A\!:= (a_{j,k})_{j \in J,k \in \N}$ such that, for each $j \in J$, there exists $k \in \N$ with $a_{j,k} > 0$,
and $0 \leq a_{j,k} \leq a_{j,k+1}$ for all $j \in J$ and $k \in \N$.
Given a real number $p \in \{0\} \cup [1,\infty)$ and such a K\"othe matrix $A$, 
the associated {\em K\"othe sequence space} $\lambda_p(A,J)$ is the Fr\'echet sequence space defined as follows:
\begin{itemize}
\item If $p \in [1,\infty)$, then 
    \[
    \lambda_p(A,J)\!:= \Big\{(x_j)_{j \in J} \in \K^J : \sum_{j \in J} |a_{j,k} x_j|^p < \infty \text{ for all } k \in \N\Big\}
    \] 
    endowed with the seminorms
    \[
    \|x\|_k\!:= \Big(\sum_{j \in J} |a_{j,k} x_j|^p\Big)^\frac{1}{p} \ \text{ for } x\!:= (x_j)_{j \in J} \in \lambda_p(A,J) \text{ and } k \in \N.
    \]
\item If $p = 0$, then 
    \[
    \lambda_p(A,J)\!:= \Big\{(x_j)_{j \in J} \in \K^J : \lim_{j \in J, |j| \to \infty} a_{j,k} x_j = 0 \text{ for all } k \in \N\Big\}
    \]
    endowed with the seminorms
    \[
    \|x\|_k\!:= \sup_{j \in J} |a_{j,k} x_j| \ \text{ for } x\!:= (x_j)_{j \in J} \in \lambda_p(A,J) \text{ and } k \in \N.
    \]
\end{itemize}
Note that the sequence $(e_j)_{j \in J}$ of canonical vectors in $\K^J$ is a basis of $\lambda_p(A,J)$ for every $p \in \{0\} \cup [1,\infty)$.

Let us adopt the convention that $\frac{0}{0} = 1$.
Given a sequence $w\!:= (w_j)_{j \in J}$ of nonzero scalars, it is well-known that the weighted backward shift
\[
B_w((x_j)_{j \in J})\!:= (w_j x_{j+1})_{j \in J}
\]
is a well-defined operator on $\lambda_p(A,J)$ if and only if, for each $k \in \N$, there exists $m \in \N$ such that
$a_{j,k} = 0$ whenever $a_{j+1,m} = 0$ ($j \in J$), and  
\begin{equation}\label{Kothe-Eq1}
\sup_{j \in J} \frac{a_{j,k} |w_j|}{a_{j+1,m}} < \infty.
\end{equation}

We refer the reader to the books \cite{GKot69,RMeiDVog97} for a detailed study of K\"othe sequence spaces.

The next result characterizes Li-Yorke chaos for unilateral weighted backward shifts.

\begin{theorem}\label{Kothe-LY-T1}
Consider a K\"othe sequence space $X\!:= \lambda_p(A,\N)$, 
where $A\!:= (a_{j,k})$ is a K\"othe matrix on $\N$ with nonzero entries and $p \in \{0\} \cup [1,\infty)$. 
Let $w\!:= (w_n)_{n \in \N}$ be a sequence of nonzero scalars such that 
the unilateral weighted backward shift $B_w$ is a well-defined operator on $X$.
Then $B_w$ is Li-Yorke chaotic if and only if there exists $k \in \N$ such that
\[
\sup\Big\{\frac{a_{i,k} |w_i \cdots w_j|}{a_{j+1,\ell}} : i,j \in \N, i \leq j\Big\} = \infty \ \ \text{ for all } \ell \in \N.
\]
\end{theorem}

Below is the corresponding result for bilateral weighted backward shifts.

\begin{theorem}\label{Kothe-LY-T2}
Consider a K\"othe sequence space $X\!:= \lambda_p(A,\Z)$, 
where $A\!:= (a_{j,k})$ is a K\"othe matrix on $\Z$ with nonzero entries and $p \in \{0\} \cup [1,\infty)$. 
Let $w\!:= (w_n)_{n \in \Z}$ be a sequence of nonzero scalars such that 
the bilateral weighted backward shift $B_w$ is a well-defined operator on $X$.
Then $B_w$ is Li-Yorke chaotic if and only if the following conditions hold: 
\begin{itemize}
\item [\rm (a)] There is an increasing sequence $(n_i)_{i \in \N}$ of positive integers such that
  \[
  \lim_{i \to \infty} a_{-n_i,k} w_{-n_i} \cdots w_{-1} = 0 \ \ \text{ for all } k \in \N.
  \]
\item [\rm (b)] There exists $k \in \N$ such that
  \[
  \sup\Big\{\frac{a_{i,k} |w_i \cdots w_j|}{a_{j+1,\ell}} : i,j \in \Z, i \leq j\Big\} = \infty \ \ \text{ for all } \ell \in \N.
  \]
\end{itemize}
\end{theorem}

Below we will prove Theorem~\ref{Kothe-LY-T2}. 
The proof of Theorem~\ref{Kothe-LY-T1} is similar (but slightly simpler) and is left to the reader.

\begin{proof}
Suppose that properties (a) and (b) hold. 
We shall prove that $B_w$ is Li-Yorke chaotic by applying Theorem~\ref{BWBSThm}.
Since (a) is equivalent to property (a) in this theorem, it remains to show that property (b) in this theorem is also true.
For this purpose, fix $k \in \N$ as in (b).
For each $n \in \N$, the continuity of $B_w^n$ implies the existence of a constant $C_n \in (0,\infty)$ and an integer $r_n \in \N$ such that
\[
\|B_w^n(x)\|_k \leq C_n \|x\|_{r_n} \ \ \text{ for all } x \in X.
\]
By applying this inequality to the canonical vectors $e_j$, we see that
\begin{equation}\label{Kothe-LY-F1}
\sup\Big\{\frac{a_{i,k} |w_i \cdots w_j|}{a_{j+1,r_n}} : i,j \in \Z, j - i = n-1\Big\} \leq C_n \ \ \text{ for all } n \in \N.
\end{equation}
Without loss of generality, we may assume that the sequence $(r_n)_{n \in \N}$ is strictly increasing.
We shall construct a sequence $(t_s)_{s \in \N}$ of positive integers and two sequences $(i_s)_{s \in \N}$ and $(j_s)_{s \in \N}$
of integers such that the following properties hold:
\begin{itemize}
\item [($\alpha$)] The sequence $(t_s)_{s \in \N}$ is strictly increasing.
\item [($\beta$)] For every $s \in \N$, we have that $i_s \leq j_s$ and $a_{i_s,k} |w_{i_s} \cdots w_{j_s}| > 2^s a_{j_s + 1,t_s}$.
\item [($\gamma$)] The sequence $(n_s)_{s \in \N}$, where $n_s\!:= j_s - i_s + 1$, is strictly increasing.
\item [($\delta$)] $j_s \neq j_t$ whenever $s \neq t$.
\end{itemize}
We begin by putting $t_1\!:= r_1$ and by choosing $i_1,j_1 \in \Z$ such that
\[
i_1 \leq j_1 \ \ \text{ and } \ \ \frac{a_{i_1,k} |w_{i_1} \cdots w_{j_1}|}{a_{j_1 + 1,t_1}} > 2.
\]
The equality in (b) for $\ell\!:= t_1$ guarantees the existence of such a pair $i_1,j_1$.
Suppose that $t_1,\ldots,t_s$, $i_1,\ldots,i_s$ and $j_1,\ldots,j_s$ have already been chosen with the desired properties.
Let
\[
C\!:= \max\{C_1,\ldots,C_{n_s+1}\}
\]
and choose $t_{s+1} \in \N$ such that
\[
t_{s+1} > \max\{t_s,r_{n_s + 1}\}.
\]
It follows from (\ref{Kothe-LY-F1}) that
\begin{equation}\label{Kothe-LY-F2}
\sup\Big\{\frac{a_{i,k} |w_i \cdots w_j|}{a_{j+1,t_{s+1}}} : i,j \in \Z, 0 \leq j - i \leq n_s\Big\} \leq C.
\end{equation}
Now we have to consider two possibilities.

\medskip\noindent
{\sc Case 1.} There exists $j_0 \in \Z$ such that
\begin{equation}\label{Kothe-LY-F3}
\sup\Big\{\frac{a_{i,k} |w_i \cdots w_{j_0}|}{a_{j_0 + 1,t_{s+1}}} : i \in \Z, i \leq j_0\Big\} = \infty.
\end{equation}

\smallskip
Then (\ref{Kothe-LY-F3}) holds with $j$ instead of $j_0$ for every $j \geq j_0$.
Hence, there exist $i_{s+1},j_{s+1} \in \Z$ such that
\[
j_{s+1} > \max\{j_1,\ldots,j_s\}, \ \ i_{s+1} \leq j_{s+1} \ \text{ and } \ 
\frac{a_{i_{s+1},k} |w_{i_{s+1}} \cdots w_{j_{s+1}}|}{a_{j_{s+1} + 1,t_{s+1}}} > \max\{C,2^{s+1}\}.
\]
Thus, (\ref{Kothe-LY-F2}) implies that $n_{s+1} = j_{s+1} - i_{s+1} + 1 > n_s$.
 
\medskip\noindent
{\sc Case 2.} For every $j \in \Z$,
\begin{equation}\label{Kothe-LY-F4}
K_j\!:= \sup\Big\{\frac{a_{i,k} |w_i \cdots w_j|}{a_{j+1,t_{s+1}}} : i \in \Z, i \leq j\Big\} < \infty.
\end{equation}

\smallskip
In this case, we apply the equality in (b) for $\ell\!:= t_{s+1}$ to obtain a pair $i_{s+1},j_{s+1} \in \Z$ such that
\[
i_{s+1} \leq j_{s+1} \ \text{ and } \ 
\frac{a_{i_{s+1},k} |w_{i_{s+1}} \cdots w_{j_{s+1}}|}{a_{j_{s+1} + 1,t_{s+1}}} > \max\{C,K_{j_1},\ldots,K_{j_s},2^{s+1}\}.
\]
Then $j_{s+1} \not\in \{j_1,\ldots,j_s\}$ and, as in Case~1, $n_{s+1} > n_s$.

\smallskip
This completes the construction of the sequences satisfying ($\alpha$)--($\delta$). Now, define
\[
x_{j_s + 1}\!:= \frac{1}{2^s a_{j_s + 1,t_s}} \ \ (s \in \N) \ \ \text{ and } \ \ x_j\!:= 0 \ \text{ if } j \in \Z \backslash \{j_s + 1 : s \in \N\}.
\]
Property ($\delta$) ensures that the sequence $x\!:= (x_j)_{j \in \Z}$ is well-defined.
Additionally, from its definition we see that $x$ satisfies
\[
\lim_{j \to \pm \infty} a_{j,\ell}\, x_j = 0
\]
for every $\ell \in \N$, so that $x \in X$ in case of $p = 0$.
Given $\ell \in \N$, property ($\alpha$) implies the existence of an integer $s_0 \in \N$ such that $t_s \geq \ell$ whenever $s > s_0$.
Therefore, if $p \in [1,\infty)$, then
\begin{align*}
\sum_{j \in \Z} |a_{j,\ell}\, x_j|^p &= \sum_{s=1}^\infty |a_{j_s + 1,\ell}\, x_{j_s + 1}|^p
  \leq \sum_{s=1}^{s_0} |a_{j_s + 1,\ell}\, x_{j_s + 1}|^p + \sum_{s=s_0+1}^\infty |a_{j_s + 1,t_s} x_{j_s + 1}|^p\\
  &= \sum_{s=1}^{s_0} |a_{j_s + 1,\ell}\, x_{j_s + 1}|^p + \sum_{s=s_0+1}^\infty \Big(\frac{1}{2^s}\Big)^p < \infty.
\end{align*}
Hence, $x \in X$ no matter what is the value of $p$.
By property ($\beta$), for any $s \in \N$,
\[
\|B_w^{n_s}(x)\|_k \geq |a_{i_s,k} w_{i_s} \cdots w_{i_s + n_s - 1} x_{i_s + n_s}| = |a_{i_s,k} w_{i_s} \cdots w_{j_s} x_{j_s + 1}| > 1.
\]
In view of property ($\gamma$), we conclude that $B_w^n(x) \not\to 0$ as $n \to \infty$.
Thus, by Theorem~\ref{BWBSThm}, $B_w$ is Li-Yorke chaotic.

Conversely, suppose that $B_w$ is Li-Yorke chaotic.
We know that property (a) holds, since it is equivalent to property (a) in Theorem~\ref{BWBSThm}.
Suppose, by contradiction, that property (b) is false.
Then, for each $k \in \N$, there exists $\ell_k \in \N$ such that
\[
C_k\!:= \sup\Big\{\frac{a_{i,k} |w_i \cdots w_j|}{a_{j+1,\ell_k}} : i,j \in \Z, i \leq j\Big\} < \infty.
\]
Take $x\!:= (x_j)_{j \in \Z} \in X$, $k \in \N$ and $n \in \N$. If $p \in [1,\infty)$, then
\[
\|B_w^n(x)\|_k^p = \sum_{j=-\infty}^\infty |a_{j,k} w_j \cdots w_{j+n-1} x_{j+n}|^p
  \leq \sum_{j=-\infty}^\infty |C_k a_{j+n,\ell_k} x_{j+n}|^p = C_k^p \|x\|_{\ell_k}^p.
\]
If $p = 0$, then
\[
\|B_w^n(x)\|_k = \sup_{j \in \Z} |a_{j,k} w_j \cdots w_{j+n-1} x_{j+n}|
  \leq \sup_{j \in \Z} |C_k a_{j+n,\ell_k} x_{j+n}| = C_k \|x\|_{\ell_k}.
\]
In any case, we see that the sequence $(B_w^n(x))_{n \in \N}$ is bounded for all $x \in X$.
Thus, $B_w$ does not admit an irregular vector, which contradicts the fact that $B_w$ is Li-Yorke chaotic.
\end{proof}

For unilateral weighted backward shifts $B_w$ on K\"othe sequence spaces, 
several properties that are equivalent to $B_w$ being Li-Yorke chaotic were obtained in
\cite[Theorem~1]{XWuPZhu13} and \cite[Theorem~3.3]{WuCheZhu14}.
Nevertheless, to the best of the authors knowledge, no characterization based solely on the entries of the K\"othe matrix
and the weights of the shift were known before.
Theorem~\ref{Kothe-LY-T1} fills this gap in the case of K\"othe sequence spaces defined by a K\"othe matrix with nonzero entries.
Moreover, Theorem~\ref{Kothe-LY-T2} gives the corresponding result in the case of bilateral shifts.

\begin{example}
Let $J\!:= \N$ or $\Z$. If $a_{j,k}\!:= 1$ for all $j \in J$ and $k \in \N$, then
\[
\lambda_0(A,J) = c_0(J) \ \ \text{ and } \ \ \lambda_p(A,J) = \ell^p(J) \text{ for } p \in [1,\infty).
\]
In this case, Theorems~\ref{Kothe-LY-T1} and~\ref{Kothe-LY-T2} recover 
Corollaries~\ref{CorLYCX1}, \ref{CorLYCX2}, \ref{CorLY2} and~\ref{CorLY3}
in the case of nonzero weights.
This provides another approach to obtain the characterizations of the Li-Yorke chaotic weighted shifts
on the classical Banach sequence spaces $c_0$ and $\ell^p$.
\end{example}

\begin{example}
Let $J\!:= \N$ or $\Z$. If $a_{j,k}\!:= (|j| + 1)^k$ for all $j \in J$ and $k \in \N$, then
\[
\lambda_1(A,J) = s(J),
\]
the {\em space of rapidly decreasing sequences on $J$}, which is a classical example of a non-normable Fr\'echet sequence space.
As an illustration, let us exhibit a weighted shift $B_w$ on $s(\Z)$ which is Li-Yorke chaotic but not hypercyclic.
Consider the weight sequence $w\!:= (w_n)_{n \in \Z}$, where $w_{-n}\!:= \frac{1}{2}$ for all $n \geq 0$, and
\[
(w_n)_{n \geq 1}\!:= \Big(\frac{1}{2},2,\frac{1}{2},\frac{1}{2},2,2,\frac{1}{2},\frac{1}{2},\frac{1}{2},2,2,2,\ldots\Big),
\]
a concatenation of blocks of the form $\big(\frac{1}{2},\ldots,\frac{1}{2},2,\ldots,2\big)$ 
with $n$ numbers $\frac{1}{2}$ followed by $n$ numbers $2$, for $n \geq 1$.
Since (\ref{Kothe-Eq1}) holds, the bilateral weighted backward shift $B_w$ is a well-defined operator on $s(\Z)$.
Since
\[
\lim_{n \to \infty} (|-n|+1)^k w_{-n} \cdots w_{-1} = \lim_{n \to \infty} \frac{(n+1)^k}{2^n} = 0,
\]
for all $k \in \N$, and
\[
\sup_{i \leq j} \frac{(|i|+1) |w_i \cdots w_j|}{(|j+1| + 1)^\ell}
\geq \sup_{n \in \N} \frac{w_{n^2+1} \cdots w_{n^2+n}}{(n^2+n+2)^\ell}
= \sup_{n \in \N} \frac{2^n}{(n^2+n+2)^\ell} = \infty,
\]
for all $\ell \in \N$, it follows from Theorem~\ref{Kothe-LY-T2} that $B_w$ is Li-Yorke chaotic.
On the other hand, by \cite[Theorem~4.13]{KGroAPer11}, $B_w$ is hypercyclic if and only if there is an increasing sequence
$(n_j)_{j \in \N}$ of positive integers such that, for any $\ell \in \Z$,
\begin{equation}\label{Kothe-LY-F5}
w_{\ell - n_j} \cdots w_{\ell - 1} e_{\ell - n_j} \to 0 \ \text{ and } \ \frac{e_{\ell + n_j}}{w_\ell \cdots w_{\ell + n_j -1}} \to 0
\ \text{ as } j \to \infty.
\end{equation}
However, in our example,
\[
\Big\|\frac{e_{1 + n}}{w_1 \cdots w_n}\Big\|_1 = \frac{n+2}{w_1 \cdots w_n} \geq n + 2 \ \ \text{ for all } n \in \N,
\]
which shows that the second condition in (\ref{Kothe-LY-F5}) fails for $\ell = 1$.
\end{example}


\section*{Acknowledgement}

The authors thank the anonymous referees for their careful reading of the manuscript and valuable suggestions. 
The first author is beneficiary of a grant within the framework of the grants for the retraining, modality Mar\'ia Zambrano, 
in the Spanish university system (Spanish Ministry of Universities, financed by the European Union, NextGenerationEU).
The first author was also partially supported by CNPq -- Project {\#}308238/2021-4, by CAPES -- Finance Code 001,
and by MCIN/AEI/10.13039/501100011033/FEDER, UE, Project PID2022-139449NB-I00.



\begin{thebibliography}{99}

\bibitem{BayDarPir18} F. Bayart, U. B. Darji and B. Pires,
    {\it Topological transitivity and mixing of composition operators}, J. Math.\ Anal.\ Appl.\ {\bf 465} (2018), no.\ 1, 125--139.
    
\bibitem{FBayEMat09} F. Bayart and \'E. Matheron, 
    {\it Dynamics of Linear Operators}, Cambridge University Press, Cambridge, 2009.

\bibitem{FBayEMat16} F. Bayart and \'E. Matheron, 
    {\it Mixing operators and small subsets of the circle}, J. Reine Angew.\ Math.\ {\bf 715} (2016), 75--123.

\bibitem{BerBonMarPer11} T. Berm\'udez, A. Bonilla, F. Mart\'inez-Gim\'enez and A. Peris,
    {\it Li-Yorke and distributionally chaotic operators}, J. Math.\ Anal.\ Appl.\ {\bf 373} (2011), no.\ 1, 83--93.

\bibitem{LBerAMon95} L. Bernal-Gonz\'alez and A. Montes-Rodr\'iguez,
    {\it Universal functions for composition operators}, Complex Variables Theory Appl.\ {\bf 27} (1995), no.\ 1, 47--56.

\bibitem{BerBonMulPer13} N. C. Bernardes Jr., A. Bonilla, V. M\"uller and A. Peris,
    {\it Distributional chaos for linear operators}, J. Funct.\ Anal.\ {\bf 265} (2013), no.\ 9, 2143--2163.

\bibitem{BerBonMulPer15} N. C. Bernardes Jr., A. Bonilla, V. M\"uller and A. Peris,
    {\it Li-Yorke chaos in linear dynamics}, Ergodic Theory Dynam.\ Systems {\bf 35} (2015), no.\ 6, 1723--1745.

\bibitem{BerBonPer20} N. C. Bernardes Jr., A. Bonilla and A. Peris,
     {\it Mean Li-Yorke chaos in Banach spaces}, J. Funct.\ Anal.\ {\bf 278} (2020), no.\ 3, Paper No.\ 108343, 31 pp.

\bibitem{BerDarPir20} N. C. Bernardes Jr., U. B. Darji and B. Pires,
    {\it Li-Yorke chaos for composition operators on $L^p$-spaces}, Monatsh.\ Math.\ {\bf 191} (2020), no.\ 1, 13--35.

\bibitem{NBerAMes20} N. C. Bernardes Jr. and A. Messaoudi,
    {\it A generalized Grobman-Hartman theorem}, Proc.\ Amer.\ Math.\ Soc.\ {\bf 148} (2020), no.\ 10, 4351--4360.

\bibitem{NBerAMes21} N. C. Bernardes Jr. and A. Messaoudi,
    {\it Shadowing and structural stability for operators}, Ergodic Theory Dynam.\ Systems {\bf 41} (2021), no.\ 4, 961--980.

\bibitem{NBerAPer24} N. C. Bernardes Jr.\ and A. Peris,
    \textit{On shadowing and chain recurrence in linear dynamics}, Adv.\ Math.\ {\bf 441} (2024), Paper No.\ 109539, 46 pp.

\bibitem{BesMenPerPui19} J. B\`es, Q. Menet, A. Peris and Y. Puig,
     {\it Strong transitivity properties for operators}, J. Differential Equations {\bf 266} (2019), no.\ 2-3, 1313--1337.

\bibitem{BonKalPer21} J. Bonet, T. Kalmes and A. Peris,
    {\it Dynamics of shift operators on non-metrizable sequence spaces}, Rev.\ Mat.\ Iberoam.\ {\bf 37} (2021), no.\ 6, 2373--2397.

\bibitem{BonDAnDarPia22} D. Bongiorno, E. D'Aniello, U. B. Darji and L. Di Piazza,
    {\it Linear dynamics induced by odometers}, Proc.\ Amer.\ Math.\ Soc.\ {\bf 150} (2022), no.\ 7, 2823--2837.
    
\bibitem{PBouJSha97} P. S. Bourdon and J. H. Shapiro,
    {\it Cyclic phenomena for composition operators}, Mem.\ Amer.\ Math.\ Soc.\ {\bf 125} (1997), no.\ 596.

\bibitem{DAnDarMai21} E. D'Aniello, U. B. Darji and M. Maiuriello,
    {\it Generalized hyperbolicity and shadowing in $L^p$ spaces}, J. Differential Equations {\bf 298} (2021), 68--94.

\bibitem{DAnDarMai22} E. D'Aniello, U. B. Darji and M. Maiuriello,
    {\it Shift-like operators on $L^p(X)$}, J. Math.\ Anal.\ Appl.\ {\bf 515} (2022), no.\ 1, Paper No.\ 126393, 13 pp.

\bibitem{EDAnMMai23} E. D'Aniello and M. Maiuriello,
    {\it On spaceability of shift-like operators on $L^p$}, J. Math.\ Anal.\ Appl.\ {\bf 526} (2023), no.\ 1, Paper No.\ 127177, 10 pp.

\bibitem{UDarBPir21} U. B. Darji and B. Pires,
    {\it Chaos and frequent hypercyclicity for composition operators}, Proc.\ Edinb.\ Math.\ Soc.\ (2) {\bf 64} (2021), no.\ 3, 513--531.
 
\bibitem{EGalAMon04} E. A. Gallardo-Guti\'errez and A. Montes-Rodr\'iguez,
    {\it The role of the spectrum in the cyclic behavior of composition operators}, Mem.\ Amer.\ Math.\ Soc.\ {\bf 167} (2004), no.\ 791.

\bibitem{GriMatMen21} S. Grivaux, \'E. Matheron and Q. Menet,
     {\it Linear dynamical systems on Hilbert spaces: typical properties and explicit examples}, 
    Mem.\ Amer.\ Math.\ Soc.\ {\bf 269} (2021), no.\ 1315.

\bibitem{KGro00} K.-G. Grosse-Erdmann, 
    {\it Hypercyclic and chaotic weighted shifts}, Studia Math.\ {\bf 139} (2000), no.\ 1, 47--68.

\bibitem{KGroAPer11} K.-G. Grosse-Erdmann and A. Peris Manguillot, 
    {\it Linear Chaos}, Springer, London, 2011.

\bibitem{TKal07} T. Kalmes,
    {\it Hypercyclic, mixing, and chaotic $C_0$-semigroups induced by semiflows},
    Ergodic Theory Dynam.\ Systems {\bf 27} (2007), no.\ 5, 1599--1631.

\bibitem{TKal19} T. Kalmes,
    {\it Dynamics of weighted composition operators on function spaces defined by local properties},
    Studia Math.\ {\bf 249} (2019), no.\ 3, 259--301.

\bibitem{GKot69} G. K\"othe,
    {\it Topological Vector Spaces, I}, Die Grundlehren der mathematischen Wissenschaften, Band 159,
    Springer-Verlag, Berlin - Heidelberg - New York, 1969.

\bibitem{MarOprPer13} F. Mart\'inez-Gim\'enez, P. Oprocha and A. Peris,
    {\it Distributional chaos for operators with full scrambled sets}, Math.\ Z.\ {\bf 274} (2013), no.\ 1-2, 603--612.

\bibitem{RMeiDVog97} R. Meise and D. Vogt,
    {\it Introduction to Functional Analysis}, Oxford Graduate Texts in Mathematics, 2,
    The Clarendon Press, Oxford University Press, Oxford - New York, 1997.

\bibitem{WRud91} W. Rudin,
    {\it Functional Analysis}, Second Edition, McGraw-Hill Inc., New York, 1991.
    
\bibitem{JSha93} J. H. Shapiro,
    {\it Composition Operators and Classical Function Theory}, Springer-Verlag, New York, 1993.

\bibitem{WuCheZhu14} X. Wu, G. Chen and P. Zhu,
    {\it Invariance of chaos from backward shift on the K\"othe sequence space}, Nonlinearity {\bf 27} (2014), no.\ 2, 271--288.

\bibitem{XWuPZhu13} X. Wu and P. Zhu,
    {\it Li-Yorke chaos of backward shift operators on K\"othe sequence spaces}, Top.\ Appl.\ \textbf{160} (2013), no.\ 7, 924--929.

\end{thebibliography}
\end{document}